\documentclass[a4paper, 10pt]{amsart}
\usepackage{amsmath}
\usepackage{amssymb, color, hyperref}

\theoremstyle{plain}
\newtheorem{theorem}{\bf Theorem}[section]

\newtheorem{corollary}[theorem]{\bf Corollary}

\theoremstyle{definition}

\newcommand{\N}{\mathbb N}
\newcommand{\Z}{\mathbb Z}

 \DeclareMathOperator{\ord}{ord}

\DeclareMathOperator{\End}{End} \DeclareMathOperator{\Min}{Min}

\renewcommand{\time}{\negthinspace \times \negthinspace}

\newcommand{\red}{{\text{\rm red}}}
\renewcommand{\t}{\, | \,}

\newcommand{\ito}{\overset\sim\longrightarrow}

\numberwithin{equation}{section}

\begin{document}

\title{A realization theorem for sets of distances}

\address{Institute for Mathematics and Scientific Computing\\ University of Graz, NAWI Graz\\ Heinrichstra{\ss}e 36\\ 8010 Graz, Austria }
\email{alfred.geroldinger@uni-graz.at}
\urladdr{http://imsc.uni-graz.at/geroldinger}

\address{Universit{\'e} Paris 13 \\ Sorbonne Paris Cit{\'e} \\ LAGA, CNRS, UMR 7539,  Universit{\'e} Paris 8\\ F-93430, Villetaneuse, France \\ and \\ Laboratoire Analyse, G{\'e}om{\'e}trie et Applications (LAGA, UMR 7539) \\ COMUE  Universit{\'e} Paris Lumi{\`e}res \\  Universit{\'e} Paris 8, CNRS \\  93526 Saint-Denis cedex, France} \email{schmid@math.univ-paris13.fr}

\author{Alfred Geroldinger  and Wolfgang A. Schmid}

\thanks{This work was supported by
the Austrian Science Fund FWF, Project Number P26036-N26 and by the ANR Project Caesar, Project Number ANR-12-BS01-0011}



\begin{abstract}
Let $H$ be an atomic monoid. The set of distances $\Delta (H)$ of $H$ is the set of all $d \in \N$ with the following property: there are irreducible elements $u_1, \ldots, u_k, v_1 \ldots, v_{k+d}$ such that $u_1 \cdot \ldots \cdot u_k=v_1 \cdot \ldots \cdot v_{k+d}$ but $u_1 \cdot \ldots \cdot u_k$ cannot be written as a product of $\ell$ irreducible elements for any $\ell \in \N$ with $k<\ell<k+d$. It is well-known (and easy to show) that, if $\Delta (H)$ is nonempty, then   $\min \Delta (H) = \gcd \Delta (H)$. In this paper we show conversely that for every finite nonempty set $\Delta \subset \N$ with $\min \Delta = \gcd \Delta$ there is a finitely generated Krull monoid $H$ such that $\Delta (H)=\Delta$.
\end{abstract}

\maketitle

%%%%%%%%%%%%%%%%%%%%%%%%%%%%%%%%%%%%%%%%%%%%%%%%%%%%%%%%%%%%%%%%%%%%%%%%%
%%                                      %%%%%%%%%%%%%%%
%%%%%%%%%%%%%%%%%%%%%%%%%%%%%%%%%%%%%%%%%%%%%%%%%%%%%%%%%%%%%%%%%%%%%%%%%
\bigskip
\section{Introduction} \label{1}
\bigskip

Sets of lengths (together with all invariants describing their structure, such as sets of distances and elasticities) are a well-studied means of describing the arithmetic structure of non-factorial monoids and domains. The first goal is to describe the arithmetical invariants in terms of algebraic invariants of the underlying structure. Then the question arises to which extent the achieved results are best possible, and answers can be given by providing monoids and domains with prescribed arithmetical invariants. For a sample of such realization results for sets of lengths, sets of elasticities, sets of catenary degrees we refer to \cite{B-C-C-K-W06, Fa-Ge17a, N-P-T-W16a,   Sc09a} and to various survey articles in \cite{C-F-G-O16}.

In the present paper we focus on  sets of distances, also called delta sets in the literature.  The set of distances $\Delta (H)$ of an atomic monoid $H$ is the set of all $d \in \N$ with the following property: there are atoms (irreducible elements) $u_1, \ldots, u_k, v_1 \ldots, v_{k+d}$ such that $u_1 \cdot \ldots \cdot u_k=v_1 \cdot \ldots \cdot v_{k+d}$ but $u_1 \cdot \ldots \cdot u_k$ cannot be written as a product of $\ell$ atoms for any $\ell \in \N$ with $k<\ell<k+d$. The monoid $H$ is called half-factorial if $\Delta (H)=\emptyset$ (clearly, factorial monoids are half-factorial). If $H$ is not half-factorial, then a simple argument shows that $\min \Delta (H) = \gcd \Delta (H)$ (\cite[Proposition 1.4.4]{Ge-HK06a}). Sets of distances are finite for transfer Krull monoids of finite type (hence in particular for Krull domains with finite class group),  weakly Krull domains with finite $v$-class group, finitely generated monoids, and others (see  \cite[Theorem 13]{Ge16c}, \cite[Theorems 3.1.4 and 3.7.1]{Ge-HK06a}, \cite{Ka16b}). The question which finite sets can actually occur as a set of distances (of any monoid or domain) was open so far. It is easy to see that every singleton can be realized as a set of distances (we recall the argument in the proof of Theorem \ref{1.1}).  One of the very few results beyond this is given in a recent  paper where it is shown that every set $\Delta$ with $|\Delta|=2$ and $\min \Delta = \gcd \Delta$ can be realized as the set of distances of a numerical monoid (\cite{Co-Ka16a}).

The arithmetic structure of a Krull monoid is completely determined by its class group and the distribution of prime divisors in the classes. Let $H$ be a Krull monoid such that every class contains a prime divisor. If the class group is finite (e.g., for rings of integers in algebraic number fields), then $\Delta (H)$ is a finite interval and if $G$ is infinite, then $\Delta (H)=\N$ (\cite{Ge-Yu12b}). The assumption on the prime divisors is crucial. Indeed,  in contrast to the above result we show in the present paper that every finite nonempty set $\Delta$ with $\min \Delta = \gcd \Delta$ can occur as the set of distances of a finitely generated Krull monoid.

\newpage
\medskip
\begin{theorem} \label{1.1}
Let $\Delta \subset \N$ be a finite nonempty set of positive integers with $\min \Delta = \gcd \Delta$. Then there is a finitely generated Krull monoid $H$ such that $\Delta (H)=\Delta$.
\end{theorem}

\medskip
The above realization theorem is not restricted to abstract Krull monoids, but any such set $\Delta$ can be realized as the set of distances of a Dedekind domain or as the set of distances of a monoid of modules. To formulate this precisely, let
$R$ be a ring,  $\mathcal C$ a class of right $R$-modules which is closed under finite direct sums, under direct summands, and under isomorphisms, and suppose that $\mathcal C$ has a set $V (\mathcal C)$ of representatives (this means that every module $M$ in $\mathcal C$ is isomorphic to a unique $[M]$ in $V (\mathcal C))$. Then $V (\mathcal C)$ becomes a commutative semigroup with operation given by $[M] + [N] = [M\oplus N]$, and it encodes all possible information about direct sum decompositions of modules in $\mathcal C$. In particular, the Krull-Remak-Schmidt-Azumaya Theorem holds for $\mathcal C$ if and only if $V (\mathcal C)$ is factorial. This semigroup-theoretical approach to the study of direct-sum decompositions of modules was pushed forward by Facchini and Wiegand, and we refer to the surveys \cite{Ba-Wi13a, Fa12a}.

\medskip
\begin{corollary} \label{1.2}
Let $\Delta \subset \N$ be a finite nonempty set of positive integers with $\min \Delta = \gcd \Delta$.
\begin{enumerate}
\item There is a Dedekind domain $R$ with finitely generated class group such that $\Delta (R)=\Delta$.

\smallskip
\item There is a ring $R$ and a class of right $R$-modules $\mathcal C$ with the above properties such that for the monoid of modules $V ( \mathcal C)$ we have $\Delta \big( V (\mathcal C) \big) = \Delta$.
\end{enumerate}
\end{corollary}

\smallskip
We gather some background and fix notation in Section \ref{2}. The proofs of Theorem \ref{1.1} and of Corollary \ref{1.2} are given in Section \ref{3}.

\bigskip
\section{Background on Krull monoids and factorizations} \label{2}
\bigskip

By a {\it monoid}, we mean a commutative semigroup which has a unit element and which satisfies the cancelation laws. For any set $P$, let $\mathcal F (P)$ be the free abelian monoid with  basis $P$ and for an element
\[
a = p_1 \cdot \ldots \cdot p_{\ell} = \prod_{p \in P} p^{\mathsf v_p (a)} \in \mathcal F (P) , \quad \text{where} \quad \ell \in \N_0 \ \text{and} \ p_1, \ldots, p_{\ell} \in P \,,
\]
we denote by $|a|= \ell = \sum_{p \in P} \mathsf v_p (a) \in \N_0$ the length of $a$.

Let $H$ be a monoid. We denote by $H^{\times}$ the group of units, by $\mathsf q (H)$ the quotient group of $H$, and by $\mathcal A (H)$ the set of atoms of $H$. The monoid $H$ is called reduced if $H^{\times} = \{1\}$, and we denote by $H_{\red}=H/H^{\times}$ the associated reduced monoid of $H$. We say that $H$ is
\begin{itemize}
\item {\it atomic} if every nonunit can be written as a finite product of atoms;

\item {\it root-closed} if $x \in \mathsf q (H)$ and $x^m \in H$ for some $m \in \N$ implies that $x \in H$.
\end{itemize}
Let $\mathsf Z (H) = \mathcal F (\mathcal A (H_{\red}))$ be the factorization monoid of $H$ and let $\pi_H \colon \mathsf Z (H) \to H_{\red}$ be the canonical epimorphism. For $a \in H$, we denote by
\begin{itemize}
\item $\mathsf Z_H (a) = \pi_H^{-1} (aH^{\times}) \subset \mathsf Z (H)$ the {\it set of factorizations} of $a$, and by

\item $\mathsf L_H (a) = \{ |z| \colon z \in \mathsf Z_H (a) \} \subset \N_0$ the {\it set of lengths} of $a$.
\end{itemize}
Thus $H$ is atomic if and only if $\mathsf Z_H (a) \ne \emptyset$ for all $a \in H$.
A monoid $H$ is a {\it Krull monoid} if it satisfies one of the following equivalent conditions (\cite[Chapter 2]{Ge-HK06a}){\rm \,:}
\begin{itemize}
\item[(a)] $H$ is completely integrally closed and satisfies the ascending chain condition on divisorial ideals.

\item[(b)] There is a free abelian monoid $F$ and a homomorphism $\varphi \colon H \to F$ such that for all $a, b \in H$ we have that $a \t b$ in $H$ if and only if $\varphi (a) \t \varphi (b)$ in $F$.
\end{itemize}
Every Krull monoid is atomic, its sets of lengths are all finite, and since it is completely integrally closed, it is root-closed.
Let $R$ be a domain. Then its multiplicative semigroup of nonzero elements $R^{\bullet} = R \setminus \{0\}$ is a monoid, and $R$ is a Krull domain if and only if $R^{\bullet}$ is a Krull monoid. Thus Property (a) reveals that every integrally closed noetherian domain is a Krull domain.

We briefly discuss a Krull monoid having a combinatorial flavor which plays a crucial role in all arithmetic studies of Krull monoids and which we will need in the proofs of Theorem \ref{1.1} and of Corollary \ref{1.2}. Let $G$ be an additively written abelian group and let $G_0 \subset G$ be a subset. For an  element $S = g_1 \cdot \ldots \cdot g_{\ell} = \prod_{g \in G_0} g^{\mathsf v_g (S)} \in \mathcal F (G_0)$ let
\[
\sigma (S) = \sum_{i=1}^{\ell}g_i = \sum_{g \in G_0} \mathsf v_g (S) g \in G \quad \text{denote  its {\it sum}} \,.
\]
Then $\mathcal B (G_0) = \{ S \in \mathcal F (G_0) \colon \sigma (S)=0 \}$ is a submonoid of $\mathcal F (G_0)$, called the {\it monoid of zero-sum sequences} over $G_0$.  Since the inclusion $\mathcal B (G_0) \hookrightarrow \mathcal F (G_0)$ satisfies Property (b), $\mathcal B (G_0)$ is a Krull monoid.
We refer the reader to the monographs \cite{Ge-HK06a, HK98} for  detailed expositions on Krull monoids.

\bigskip
\section{Proof of the Main Results} \label{3}
\bigskip

Before starting the proof of Theorem \ref{1.1}, we recall some basic facts that are needed in the proof. A submonoid $H' \subset H$ of a monoid $H$ is  called divisor-closed if $a\in H'$ and $b \in H $ with $b \mid  a $ (in $H$) implies that $b \in H'$. If $H' \subset H$ is a divisor-closed submonoid, then for each $a \in H'$ one has $\mathsf{Z}_{H'}(a)= \mathsf{Z}_{H}(a)$, in particular $\mathsf{L}_{H'}(a)= \mathsf{L}_{H}(a)$ and $\Delta(H') \subset \Delta(H)$.

If $H = H_1 \times H_2$ is the product of two reduced atomic monoids $H_1, H_2$, then $H_1$ and $H_2$ are divisor-closed submonoids whence
$\Delta(H_1) \cup \Delta(H_2) \subset \Delta (H)$. In general, this inclusion is strict. This is why a direct product construction (as used in a realization result for sets of catenary degrees and others, \cite[Proposition 3.2]{Fa-Ge17a}) cannot be used to construct arbitrary sets of distances.

\begin{proof}[Proof of Theorem \ref{1.1}]
Let $\Delta \subset \N$ be a finite subset with $\min \Delta = \gcd \Delta$. We proceed by induction on $|\Delta|$.
Suppose that $\Delta = \{d\}$. Let $G$ be an additive abelian group having an element  $g \in G$ with $\ord (g)=d+2$, and let $H = \mathcal B ( \{g, -g\})$. We set $v_1 = g^{d+2}$, $v_2=(-g)^{d+2}$, and $v_3=(-g)g$. Then $\mathcal A (\mathcal B (G_0))= \{v_1, v_2, v_3\}$, $v_1v_2=v_3^{d+2}$, and $H$ is a reduced finitely generated Krull monoid with $\Delta (H) = \{d\}$.

Now suppose that $|\Delta|>1$ and that the assertion holds for all sets under consideration which are strictly smaller than $|\Delta|$. We set $d_1 = \min \Delta$, $d_2 = \max \Delta$, $\Delta_1 = \Delta \setminus \{d_2\}$, and $\Delta_2 = \{d_2\}$. Then $\min \Delta_1 = \gcd \Delta_1 = d_1$ and $\min \Delta_2 = \gcd \Delta_2$. Thus by the induction hypothesis there are finitely generated Krull monoids $H_1$ and $H_2$ such that $\Delta (H_i) = \Delta_i$ for $i \in [1,2]$. Without restriction we may suppose that $H_1$ and $H_2$ are reduced and that $\mathcal A (H_2) = \{v_1, v_2, v_{3} \}$ with $v_1v_2=v_3^{d_2+2}$ as constructed for the case $|\Delta|=1$.
We set $\mathcal A (H_1) = \{u_1, \ldots, u_k\}$ and observe that $\mathcal A (H_1 \time H_2) = \mathcal A (H_1) \uplus \mathcal A (H_2) = \{u_1, \ldots, u_k, v_1, v_2,v_3\}$. We define
\[
\begin{aligned}
\pi \colon \N_0^{k+3} & \ito \qquad \mathsf Z (H_1\time H_2) & \longrightarrow H_1 \time H_2 \\
\boldsymbol m & \mapsto z = u_1^{m_1} \cdot \ldots \cdot u_k^{m_k} v_1^{n_1}v_2^{n_2}v_3^{n_3} & \mapsto \pi_{H_1 \time H_2} (z)
\end{aligned}
\]
and
\[
\begin{aligned}
\Omega = \{ \boldsymbol m = (m_1, \ldots, m_k, n_1,n_2,n_3) \in \N_0^{k+3} \colon &   (n_1,n_2,n_3) \ne (0,0,0), (m_1, \ldots, m_k) \ne (0, \ldots, 0), \ \text{and} \\
 & |\boldsymbol m| < \max  \mathsf L_{H_1\time H_2} ( \pi (\boldsymbol m) ) - d_1 \}
\end{aligned}
\]
We note that this set is in fact nonempty, as for example $\mathsf L_{H_1\time H_2}(u_1 v_1v_2)= \{3, 3 + d_2\}$.
The set of minimal elements  $\Min (\Omega) \subset \Omega$ (with respect to the usual partial order) is  finite by Dickson's  Theorem (\cite[Theorem 1.5.3]{Ge-HK06a}).  For each $\boldsymbol m \in \Min (\Omega)$ and each
\[
\ell \in [|{\boldsymbol m}|+d_1, \max \mathsf L_{H_1\time H_2} (\pi (\boldsymbol m) ) - d_1] \cap (|{\boldsymbol m}| + d_1\N)
\]
we consider a set
$P_{\ell}^{\boldsymbol m} = \{p^{\boldsymbol m}_{\ell,2}, \ldots, p^{\boldsymbol m}_{\ell, \ell}\}$.
We define $\Omega' \subset \Min (\Omega) \time \mathbb{N}$ to be the set of all  $(\boldsymbol m, \ell)$ where $\boldsymbol m \in \Min (\Omega)$ and $\ell$ as above, and note that $ \Omega'$ is finite too.
For each  $(\boldsymbol m, \ell) \in \Omega'$, we define
\[
p^{\boldsymbol m}_{\ell,1} = \pi (\boldsymbol m)(p^{\boldsymbol m}_{\ell, 2} \cdot \ldots \cdot p^{\boldsymbol m}_{\ell, \ell})^{-1} \in \mathsf q (H_1 \time H_2) \times  \prod_{(\boldsymbol m, \ell) \in \Omega'} \mathsf q (\mathcal F (P^{\boldsymbol m}_{\ell})) \,,
\]
and set $P_{\ell}^{\boldsymbol m, \ast} = \{p^{\boldsymbol m}_{\ell,1}, \ldots, p^{\boldsymbol m}_{\ell, \ell}\}$.
Let
\[
H  \subset \mathsf q (H_1 \time H_2) \times  \prod_{(\boldsymbol m, \ell) \in \Omega'} \mathsf q (\mathcal F (P^{\boldsymbol m}_{\ell}))
\]
be the submonoid generated by $A = \mathcal A (H_1 \time H_2) \uplus \biguplus_{(\boldsymbol m, \ell) \in \Omega'} P_{\ell}^{\boldsymbol m, \ast}$. Thus $H$ is a reduced finitely generated monoid with quotient group $\mathsf q (H) = \mathsf q (H_1 \time H_2) \times  \prod_{(\boldsymbol m, \ell) \in \Omega'} \mathsf q (\mathcal F (P^{\boldsymbol m}_{\ell}))$.  Since $A$ is a minimal generating set (with respect to inclusion), it follows by \cite[Proposition 1.1.7]{Ge-HK06a} that $A$ is the set of atoms of $H$.
If $i \in [1,2]$, then $H_i \subset H$ is a divisor-closed submonoid and hence if $a \in H_i$, then $\mathsf Z_{H_i} (a) = \mathsf Z_H (a)$ and thus $\mathsf L_{H_i} (a) = \mathsf L_H (a)$.
Clearly, $H_1\time H_2 \subset H$ is not divisor-closed, and  hence if $a \in H_1 \time H_2 \setminus (H_1 \cup H_2)$, then $\mathsf L_H (a)$ and $\mathsf L_{H_1\time H_2}(a)$ need not  be equal. Consider a product of the form
\[
a= \prod_{\nu=1}^k u_{\nu}^{r_{\nu}} \prod_{\nu=1}^3v_{\nu}^{s_{\nu}} \prod_{(\boldsymbol m, \ell) \in \Omega'} \prod_{\nu=1}^{\ell} (p_{\ell, \nu}^{\boldsymbol m})^{t_{\boldsymbol m, \ell, \nu}} \,,
\]
where all exponents $r_{\nu}, s_{\nu}$, and $t_{\boldsymbol m, \ell, \nu}$ are non-negative integers. Then, by construction, we have
\begin{equation} \label{characterization}
a \in H_1 \time H_2 \quad \text{ if and only if } \quad  t_{\boldsymbol m, \ell, 1} = \ldots = t_{\boldsymbol m, \ell, \ell} \ \text{ for all} \  (\boldsymbol m, \ell) \in \Omega' \,.
\end{equation}

In the remainder of the proof we show the following three assertions. The first one verifies a condition that is important for technical reasons and the latter two then establish what we intended to show.
\begin{enumerate}
\item[{\bf A1.}\,] If $a_1 \in H_1$ and $a_2 \in H_2$, then $\max \mathsf L_H (a_1a_2) = \max \mathsf L_{H_1 \time H_2} (a_1a_2)$.

\smallskip

\item[{\bf A2.}\,] $\Delta (H)=\Delta$.

\smallskip

\item[{\bf A3.}\,] $H$ is a Krull monoid.
\end{enumerate}

\medskip
\noindent
{\it Proof of \,{\bf A1}}.\, Let $a_1 \in H_1$, $a_2 \in H_2$, and $z\in \mathsf{Z}_H(a_1a_2)$ with $|z|=\max \mathsf L_H (a_1a_2)$. We need to show that there exists some  $z'\in \mathsf{Z}_{H_1 \times H_2}(a_1a_2)$ such that $|z|= |z'|$.
Suppose that
\[
z= \prod_{\nu=1}^k u_{\nu}^{r_{\nu}} \prod_{\nu=1}^3v_{\nu}^{s_{\nu}} \prod_{(\boldsymbol m, \ell) \in \Omega'} \prod_{\nu=1}^{\ell} (p_{\ell, \nu}^{\boldsymbol m})^{t_{\boldsymbol m, \ell, \nu}} \,,
\]
where all exponents $r_{\nu}, s_{\nu}$, and $t_{\boldsymbol m, \ell, \nu}$ are non-negative integers.
If all $t_{\boldsymbol m, \ell, \nu}$ are equal to $0$, then $z \in \mathsf{Z}_{H_1 \times H_2}(a_1a_2)$, and the claim holds with $z'=z$.
Assume to the contrary that there are $(\boldsymbol m', \ell') \in \Omega'$ and $\nu' \in [1, \ell']$ such that $t_{\boldsymbol m', \ell', \nu'} > 0$.
Since we consider a factorization of the element  $a_1a_2  \in H_1 \time H_2$, we infer by \eqref{characterization} that
$t_{\boldsymbol m', \ell', \nu'} = t_{\boldsymbol m', \ell', 1} = \ldots = t_{\boldsymbol m', \ell', \ell'}$.
By definition of $\Omega'$ we have $\max \mathsf{L}_{H_1 \time H_2}( \pi(\boldsymbol m')) \ge \ell' +d_1$.  If $y \in \mathsf{Z}_{H_1 \time H_2}( \pi(\boldsymbol m'))$ is a factorization of maximal length, then
\[
 \prod_{\nu=1}^{\ell'} (p_{\ell', \nu}^{\boldsymbol m'})^{-1}\, z \, y  \in \mathsf Z_H (a_1a_2)
\]
is a factorization of $a_1a_2$ of length greater than or equal to $|z|+d_1$, a contradiction to $|z| = \max \mathsf L_H (a_1a_2)$.

\medskip
\noindent
{\it Proof of \,{\bf A2}}.\,
Since $H$ is finitely generated, \cite[Theorem 3.1.4]{Ge-HK06a} implies that $\Delta (H)$ is finite and by construction we have $\Delta (H) \subset d_1 \N$.
Since, for $i \in [1,2]$, $H_i \subset H$ is a divisor-closed submonoid, it follows that $\Delta_i = \Delta (H_i) \subset \Delta (H)$ and hence $\Delta = \Delta_1 \cup \Delta_2 \subset \Delta (H)$. To verify the reverse inclusion, let $a \in H$ be given.
We choose a factorization  $z \in \mathsf{Z}_H (a)$ with $|z| <  \max \mathsf L_H (a)$ and show that the distance from $|z|$ to the next largest element in $\mathsf L_H (a)$ lies in $\Delta $. More formally, let $d = \min \{ k - |z| \colon k \in \mathsf L_H (a) \ \text{with} \ k>|z| \}$ and we assert that $d \in \Delta$.
We write $z$ in the form
$z=z_1z_2 z_0$, where $z_i \in \mathcal{F}(\mathcal{A}(H_i))$ for $i \in [1,2]$ and  $z_0 \in \mathcal{F}(   \biguplus_{(\boldsymbol m, \ell) \in \Omega'}  P_{\ell}^{\boldsymbol m, \ast} )$, and distinguish two cases.

\smallskip
\noindent
CASE 1:  There exists $(\boldsymbol m, \ell) \in \Omega'$ such that $\prod_{\nu = 1}^{\ell}p_{\ell, \nu}^{\boldsymbol m} \mid z$ in $\mathsf Z (H)$.

We show that  $d = d_1$. If $\ell < \max \mathsf L_{H_1 \time H_2} (\pi (\boldsymbol m) ) -d_1 $, then
 $(\boldsymbol m, \ell + d_1) \in \Omega'$,
and
\[
(\prod_{\nu = 1}^{\ell}p_{\ell, \nu}^{\boldsymbol m})^{-1}z(\prod_{\nu = 1}^{\ell+d_1}p_{\ell+d_1, \nu}^{\boldsymbol m}) \ \in \mathsf Z_H (a)
\]
is a factorization of $a$ of length $|z|+ (\ell+ d_1) -\ell =|z|+d_1$. Since $ \Delta (H) \subset d_1 \N$, it follows that $d=d_1$.

If $\ell   = \max \mathsf L_{H_1 \time H_2} (\pi (\boldsymbol m) ) -d_1 $, then
we choose  $y \in \mathsf{Z}_{H_1\times H_2}(\pi (\boldsymbol m ) )$ with  $|y|= \max \mathsf L_{H_1 \time H_2} (\pi (\boldsymbol m) )$.
Then
\[
(\prod_{\nu = 1}^{\ell}p_{\ell, \nu}^{\boldsymbol m})^{-1}z y \ \in \mathsf Z_H (a)
\]
is a factorization of $a$ of length $ |z|+ d_1$. Since  $ \Delta (H) \subset d_1 \N$, it follows that $d=d_1$.

\smallskip
\noindent
CASE 2: There exists no $(\boldsymbol m, \ell) \in \Omega'$ such that
$\prod_{\nu = 1}^{\ell}p_{\ell, \nu}^{\boldsymbol m} \mid z$ in $\mathsf Z (H)$.

For $i \in [1,2]$ we set $a_i = \pi_{H_i}(z_i)$, and we start with the following assertion.

\begin{enumerate}
\item[{\bf A4.}\,] $\mathsf{Z}_{H} (a_1a_2) z_0 = \mathsf{Z}_H (a)$.
\end{enumerate}

\noindent
{\it Proof of \,{\bf A4}}.\, Obviously, we have $\mathsf{Z}_{H} (a_1a_2) z_0 \subset \mathsf{Z}_H (a)$. To show the reverse inclusion, let
 $z' \in \mathsf{Z}_H (a)$. We write $z'$ in the form  $z' = z_1'z_2'z_0'$ where $z_i' \in \mathcal{F}(\mathcal{A}(H_i))$ for $i \in [1,2]$ and  $z_0' \in \mathcal{F}(   \biguplus_{(\boldsymbol m, \ell) \in \Omega'}  P_{\ell}^{\boldsymbol m, \ast} )$. Clearly, it is sufficient to show  that $z_0 \mid z_0'$.
For each $({\boldsymbol m}, \ell) \in \Omega'$ we define
\[
z(\boldsymbol m, \ell) = \prod_{\nu=1}^{\ell} (p_{\ell, \nu}^{\boldsymbol m})^{k_{\nu}} \quad \text{resp.} \quad z' (\boldsymbol m, \ell) = \prod_{\nu=1}^{\ell} (p_{\ell, \nu}^{\boldsymbol m})^{k_{\nu}'}
\]
where $k_{\nu}$ resp. $k_{\nu}'$ is the multiplicity of $p_{\ell, \nu}^{\boldsymbol m}$ in $z_0$ resp. $z_0'$ for all $\nu \in [1, \ell]$. The assumption of CASE 2 implies that there is a $\nu \in [1, \ell]$ such that $k_{\nu}=0$. We claim that $z(\boldsymbol m, \ell) \t z' (\boldsymbol m, \ell)$ for each $({\boldsymbol m}, \ell) \in \Omega'$. Once this is done, then it follows that $z_0 \mid z_0'$.

Let $({\boldsymbol m}, \ell) \in \Omega'$ and let $\phi^{\boldsymbol m}_{\ell} \colon \mathsf{q}(H)\to \mathsf{q}(\mathcal{F}(P_{\ell}^{\boldsymbol m}))$ denote the canonical projection.
We note that $\phi^{\boldsymbol m}_{\ell}(u)=1$ for $u \in \mathcal{A}(H) \setminus P_{\ell}^{\boldsymbol m, \ast}$ and
$\phi^{\boldsymbol m}_{\ell} (p_{\ell, \nu}^{\boldsymbol m})=p_{\ell, \nu}^{\boldsymbol m}$ for $\nu \in [2, \ell]$  and $\phi^{\boldsymbol m}_{\ell} (p_{\ell, 1}^{\boldsymbol m})=\prod_{\nu=2}^{\ell'} (p_{\ell, \nu}^{\boldsymbol m})^{-1}$.
It follows that
\[
\phi^{\boldsymbol m}_{\ell}(a)=\prod_{\nu=2}^{\ell} (p_{\ell, \nu}^{\boldsymbol m})^{k_{\nu} - k_1} \quad \text{and also}  \quad \phi^{\boldsymbol m}_{\ell} (a) = \prod_{\nu=2}^{\ell} (p_{\ell, \nu}^{\boldsymbol m})^{k_{\nu}'-k_1'} \,.
\]
Consequently $ k_{\nu} - k_1 = k_{\nu}' - k_1' $ for each $\nu \in [2, \ell]$.
It follows that  $k_{\nu}' = k_{\nu} + (k_1'-k_1)$ for each  $\nu \in [1, \ell]$.
Since all $k_{\nu}, k_{\nu}'$ are non-negative and at least one of the $k_{\nu}$ is equal to $0$, it follows that $k_1'-k_1 \ge 0$, and hence $z(\boldsymbol m, \ell) \t z' (\boldsymbol m, \ell)$.  \qed{(Proof of {\bf  A4})}

\smallskip
If $a_2 = 1$, then $ a_1a_2 \in H_1$.  As $H_1 \subset H$ is divisor-closed, it follows that $\Delta ( \mathsf L_H (a)) = \Delta ( \mathsf L_{H_1} (a_1))$ and hence $d \in \Delta_1 \subset \Delta$.
If $a_1 = 1$, then the analogous  argument yields $d \in \Delta_2 \subset \Delta$.
Now we  assume that $a_1\ne 1$,  $a_2\ne 1$, and assert that  $d=d_1$.
If $|z| -|z_0|=|z_1 z_2| =  \max \mathsf{L}_{H}(a_1a_2) $, then $\mathsf{Z}_H(a) = \mathsf{Z}_{H} (a_1a_2) z_0$ and {\bf A1} imply that $|z| = \max \mathsf{L}_H (a)$, a contradiction.
If $|z| -|z_0|=|z_1 z_2| \in [\max \mathsf{L}_{H} (a_1a_2) -d_1, \max \mathsf{L}_{H} (a_1a_2) -1]$, then  $d=d_1$ by the minimality of $d_1$.

It remains to consider the case where $|z_1 z_2| < \max \mathsf{L}_{H}(a_1a_2)  - d_1$.
Let $\boldsymbol{m} =(m_1, \ldots, m_k, n_1,n_2,n_3) \in \N_0^{k+3}$ such that $z_1z_2 = u_1^{m_1} \cdot \ldots \cdot u_k^{m_k}v_1^{n_1}v_2^{n_2}v_3^{n_3}$.
Since $a_1\ne 1$, $a_2\ne1$, and (by {\bf A1})
\[
|\boldsymbol m|= |z_1 z_2| < \max \mathsf{L}_{H}(a_1a_2)  - d_1 = \max \mathsf{L}_{H_1 \time H_2}(a_1a_2)  - d_1 \,,
\]
it follows that $\boldsymbol{m} \in \Omega$.
Let $\boldsymbol{m}' \in \Min (\Omega)$ with  $\boldsymbol{m}' \le \boldsymbol m$, and let $z'$ denote
the respective factorization. Then $z'\mid z_1z_2$ in $\mathsf Z (H)$,  $(\boldsymbol{m}', |z'|+ d_1)\in  \Omega'$, and
\[
z'^{-1} z_1z_2 (\prod_{\nu=1}^{|z'|+ d_1}p_{|z'|+d_1, \nu}^{\boldsymbol m}) \ \in \mathsf Z_H (a_1a_2)
\]
is a factorization of $a_1a_2$ of length $|z_1z_2| + d_1$.  Thus we obtain a factorization of $a$ of length $|z|+ d_1$  which implies $d=d_1$.

\medskip
\noindent
{\it Proof of \,{\bf A3}}.\, Since $H$ is a reduced and finitely generated monoid,  it is sufficient to verify that $H$ is root-closed by \cite[Theorem 2.7.14]{Ge-HK06a}. Let $x \in \mathsf{q}(H)$, say
\[
x = y \prod_{(\boldsymbol m, \ell) \in \Omega'} \prod_{\nu=2}^{\ell} (p_{\ell, \nu}^{\boldsymbol m})^{r_{\boldsymbol m, \ell, \nu}} \ \in \mathsf q (H)
\]
where $y \in \mathsf q (H_1 \time H_2)$ and all exponents $r_{\boldsymbol m, \ell, \nu} \in \Z$. Suppose there is an $m \in \N$ such that $x^m \in H$, say
\[
x^m = b \prod_{(\boldsymbol m, \ell) \in \Omega'} \prod_{\nu=1}^{\ell} (p_{\ell, \nu}^{\boldsymbol m})^{s_{\boldsymbol m, \ell, \nu}} \ \in  H
\]
where $b \in H_1 \time H_2$ and all exponents $s_{\boldsymbol m, \ell, \nu} \in \N_0$; note that here we use $p_{\ell, 1}^{\boldsymbol m}$, too. We have to show that $x \in H$. Clearly, we have, in $\mathsf{q}(H)$,
\[
\begin{aligned}
y^m \prod_{(\boldsymbol m, \ell) \in \Omega'} \prod_{\nu=2}^{\ell} (p_{\ell, \nu}^{\boldsymbol m})^{m r_{\boldsymbol m, \ell, \nu}} \ & = \
b \prod_{(\boldsymbol m, \ell) \in \Omega'} \prod_{\nu=1}^{\ell} (p_{\ell, \nu}^{\boldsymbol m})^{s_{\boldsymbol m, \ell, \nu}} \\
 & = b  \prod_{(\boldsymbol m, \ell) \in \Omega'}    \pi (\boldsymbol m)^{s_{\boldsymbol m, \ell, 1}}  \prod_{(\boldsymbol m, \ell) \in \Omega'} \prod_{\nu=2}^{\ell} (p_{\ell, \nu}^{\boldsymbol m})^{s_{\boldsymbol m, \ell, \nu} - s_{\boldsymbol m, \ell, 1}} \,,
\end{aligned}
\]
where the second equation follows from the fact that $p_{\ell, 1}^{\boldsymbol m} = \pi (\boldsymbol m)(\prod_{\nu=2}^{\ell} p_{\ell, \nu}^{\boldsymbol m})^{-1}$ for all $(\boldsymbol m, \ell) \in \Omega'$.
Since $b^{\ast} := b  \prod_{(\boldsymbol m, \ell) \in \Omega'}    \pi (\boldsymbol m)^{s_{\boldsymbol m, \ell, 1}} \in H_1 \time H_2 \subset \mathsf{q}( H_1 \time H_2)$, it follows that
\begin{equation} \label{wichtig}
m r_{\boldsymbol m, \ell, \nu} =   s_{\boldsymbol m, \ell, \nu} - s_{\boldsymbol m, \ell, 1} \ \text{ for all indices and} \quad  y^m =b^{\ast} \in H_1 \time H_2 \,.
\end{equation}
Since $H_1 \time H_2$ is a Krull monoid and hence root-closed, we infer that $y \in H_1 \time H_2$.
For each $(\boldsymbol m, \ell) \in \Omega'$, we choose an index $\nu' \in [1, \ell]$ such that \[
r_{\boldsymbol m, \ell, \nu'} = \min \{ r_{\boldsymbol m, \ell, \nu} \colon \nu \in [1, \ell] \}\,,
\]
and let $\Omega'' \subset \Omega'$ be the the subset of all $(\boldsymbol m, \ell) \in \Omega'$ for which  $r_{\boldsymbol m, \ell, \nu'} < 0$. We write $x$ as a product of four factors and show that each factor lies in $H$ which implies $x \in H$.
Indeed we have
\begin{align*}
x & = y \prod_{(\boldsymbol m, \ell) \in \Omega'} \prod_{\nu=2}^{\ell} (p_{\ell, \nu}^{\boldsymbol m})^{r_{\boldsymbol m, \ell, \nu}}\\
  & = (y \prod_{(\boldsymbol m, \ell) \in \Omega''}  \pi (\boldsymbol m)^{r_{\boldsymbol m, \ell, \nu'}})
            (\prod_{(\boldsymbol m, \ell) \in \Omega''}  \pi (\boldsymbol m)^{-r_{\boldsymbol m, \ell, \nu'}})
            \prod_{(\boldsymbol m, \ell) \in \Omega'} \prod_{\nu=2}^{\ell} (p_{\ell, \nu}^{\boldsymbol m})^{r_{\boldsymbol m, \ell, \nu}  } \\
   & = (y \prod_{(\boldsymbol m, \ell) \in \Omega''}  \pi (\boldsymbol m)^{r_{\boldsymbol m, \ell, \nu'}})
            (\prod_{(\boldsymbol m, \ell) \in \Omega''} \prod_{\nu=1}^{\ell} (p_{\ell, \nu}^{\boldsymbol m})^{-r_{\boldsymbol m, \ell, \nu'}  })
            \prod_{(\boldsymbol m, \ell) \in \Omega'} \prod_{\nu=2}^{\ell} (p_{\ell, \nu}^{\boldsymbol m})^{r_{\boldsymbol m, \ell, \nu}  }\\
 & = \underbrace{(y \prod_{(\boldsymbol m, \ell) \in \Omega''}  \pi (\boldsymbol m)^{r_{\boldsymbol m, \ell, \nu'}})}_{\text{(i)}}
     \underbrace{(\prod_{(\boldsymbol m, \ell) \in \Omega''}  (p_{\ell, 1}^{\boldsymbol m})^{-r_{\boldsymbol m, \ell, \nu'}})}_{\text{(ii)}}
     \underbrace{\prod_{(\boldsymbol m, \ell) \in \Omega''} \prod_{\nu=2}^{\ell} (p_{\ell, \nu}^{\boldsymbol m})^{r_{\boldsymbol m, \ell, \nu} -r_{\boldsymbol m, \ell, \nu'} }}_{\text{(iii)}}
     \underbrace{\prod_{(\boldsymbol m, \ell) \in \Omega' \setminus \Omega''} \prod_{\nu=2}^{\ell} (p_{\ell, \nu}^{\boldsymbol m})^{r_{\boldsymbol m, \ell, \nu}  }}_{\text{(iv)}}
\end{align*}
We consider each of the four products individually:
\begin{itemize}
\smallskip
\item[(i)] Since  for all indices $s_{\boldsymbol m, \ell, 1} + m r_{\boldsymbol m, \ell, \nu}  =   s_{\boldsymbol m, \ell, \nu} \ge 0 $ by \eqref{wichtig}, we obtain that
\[
\begin{aligned}
(y \prod_{(\boldsymbol m, \ell) \in \Omega''}    \pi (\boldsymbol m)^{r_{\boldsymbol m, \ell, \nu'}})^m & =
b^{\ast } \prod_{(\boldsymbol m, \ell) \in \Omega''}    \pi (\boldsymbol m)^{m r_{\boldsymbol m, \ell, \nu'}}   \\
& =   b \prod_{(\boldsymbol m, \ell) \in \Omega'} \pi (\boldsymbol m)^{s_{\boldsymbol m, \ell, 1}}
\prod_{(\boldsymbol m, \ell) \in \Omega''}    \pi (\boldsymbol m)^{m r_{\boldsymbol m, \ell, \nu'}} \\
& = b \prod_{(\boldsymbol m, \ell) \in \Omega' \setminus \Omega''} \pi (\boldsymbol m)^{s_{\boldsymbol m, \ell, 1}}
\prod_{(\boldsymbol m, \ell) \in \Omega''}    \pi (\boldsymbol m)^{s_{\boldsymbol m, \ell, 1}+m r_{\boldsymbol m, \ell, \nu'}} \in H_1 \time H_2 \,.
\end{aligned}
\]
Since $H_1 \times H_2$ is a Krull monoid and hence root-closed, it follows that
\[
y \prod_{(\boldsymbol m, \ell) \in \Omega''}    \pi (\boldsymbol m)^{r_{\boldsymbol m, \ell, \nu'}} \in H_1 \time H_2 \subset H \,.
\]

\smallskip
\item[(ii)]  This product lies in $H$ since all  $p_{\ell, 1}^{\boldsymbol m}$ are in $H$ and the exponent $-r_{\boldsymbol m, \ell, \nu'}$ is positive.

\smallskip
\item[(iii)]  This product lies in $H$ since all all $p_{\ell, \nu}^{\boldsymbol m}$ are in $H$ and all exponents $r_{\boldsymbol m, \ell, \nu} -r_{\boldsymbol m, \ell, \nu'}$ are non-negative by the minimality of $r_{\boldsymbol m, \ell, \nu'}$.

\smallskip
\item[(iv)] This product lies in $H$ since all $p_{\ell, \nu}^{\boldsymbol m}$ are in $H$ and the exponents  $r_{\boldsymbol m, \ell, \nu} $ are non-negative for all $(\boldsymbol m, \ell) \in \Omega' \setminus \Omega''$   by definition of $\Omega''$.
\end{itemize}
\end{proof}

\bigskip
\begin{proof}[Proof of Corollary \ref{1.2}]
1. Let $\Delta \subset \N$ be a finite nonempty set of positive integers with $\min \Delta = \gcd \Delta$. By Theorem \ref{1.1} there exists a finitely generated Krull monoid $H$ such that $\Delta (H) = \Delta$.
Let $G$ denote the class group of $H$ and let $G_0 \subset G$ be the set of classes containing prime divisors. Since $H$ is a finitely generated monoid, $G$ is a finitely generated abelian group. By \cite[Theorem 3.4.10]{Ge-HK06a}, there is a homomorphism $\boldsymbol \beta \colon H \to \mathcal B (G_0)$ such that $\Delta (H) = \Delta \big( \mathcal B  (G_0) \big)$.
By Claborn's Realization Theorem (\cite[Theorem 3.7.8]{Ge-HK06a}) there is a Dedekind domain $R$ and an isomorphism $\Phi \colon G \to \mathcal C (R)$ such that $G_0$ is mapped onto the subset $G_P$ of the class group $\mathcal C (R)$ which contains prime ideals. Since $R$ is a Dedekind domain, its multiplicative monoid $R^{\bullet} = R \setminus \{0\}$ is a Krull monoid and again by \cite[Theorem 3.4.10]{Ge-HK06a}  there is a homomorphism $\theta \colon R^{\bullet} \to \mathcal B (G_P)$ such that $\Delta (R^{\bullet}) = \Delta \big( \mathcal B (G_P) \big)$. Thus we obtain that
\[
\Delta = \Delta (H) = \Delta \big( \mathcal B  (G_0) \big) = \Delta \big( \mathcal B  (G_P) \big) = \Delta (R^{\bullet}) \,.
\]

\smallskip
2. Let $R$ be a ring,  $\mathcal C$ a class of right $R$-modules which is closed under finite direct sums, under direct summands, and under isomorphisms, and suppose that $\mathcal C$ has a set $V (\mathcal C)$ of representatives. If the endomorphism ring $\End_R (M)$ is semilocal for all modules $M$ in $\mathcal C$, then the monoid $V (\mathcal C)$ is a reduced Krull monoid by a theorem of Facchini (\cite[Theorem 3.4]{Fa02}. Conversely, every reduced Krull monoid is isomorphic to a monoid of modules $V (\mathcal C)$ by a realization theorem of Facchini and Wiegand (\cite[Theorem 2.1]{Fa-Wi04}). Their result together with Theorem \ref{1.1} implies the assertion.
\end{proof}

\providecommand{\bysame}{\leavevmode\hbox to3em{\hrulefill}\thinspace}
\providecommand{\MR}{\relax\ifhmode\unskip\space\fi MR }
\providecommand{\MRhref}[2]{%
  \href{http://www.ams.org/mathscinet-getitem?mr=#1}{#2}
}
\providecommand{\href}[2]{#2}


\begin{thebibliography}{10}

\bibitem{Ba-Wi13a}
N.R. Baeth and R.~Wiegand, \emph{Factorization theory and decomposition of
  modules}, Amer. Math. Monthly \textbf{120} (2013), 3 -- 34.

\bibitem{B-C-C-K-W06}
P.~Baginski, S.T. Chapman, C.~Crutchfield, K.G. Kennedy, and M.~Wright,
  \emph{Elastic properties and prime elements}, Result. Math. \textbf{49}
  (2006), 187 -- 200.

\bibitem{C-F-G-O16}
S.T. Chapman, M.~Fontana, A.~Geroldinger, and B.~Olberding (eds.),
  \emph{Multiplicative {I}deal {T}heory and {F}actorization {T}heory},
  Proceedings in Mathematics and Statistics, vol. 170, Springer, 2016.

\bibitem{Co-Ka16a}
S.~Colton and N.~Kaplan, \emph{The realization problem for delta sets of
  numerical monoids}, J. Commut. Algebra, to appear.

\bibitem{Fa02}
A.~Facchini, \emph{Direct sum decomposition of modules, semilocal endomorphism
  rings, and {K}rull monoids}, J. Algebra \textbf{256} (2002), 280 -- 307.

\bibitem{Fa12a}
\bysame, \emph{Direct-sum decompositions of modules with semilocal endomorphism
  rings}, Bull. Math. Sci. \textbf{2} (2012), 225 -- 279.

\bibitem{Fa-Wi04}
A.~Facchini and R.~Wiegand, \emph{Direct-sum decomposition of modules with
  semilocal endomorphism rings}, J. Algebra \textbf{274} (2004), 689 -- 707.

\bibitem{Fa-Ge17a}
Y.~Fan and A.~Geroldinger, \emph{Minimal relations and catenary degrees in
  {K}rull monoids}, J. Commut. Algebra, to appear.


\bibitem{Ge16c}
A.~Geroldinger, \emph{Sets of lengths}, Amer. Math. Monthly \textbf{123} (2016), 960 --  988.

\bibitem{Ge-HK06a}
A.~Geroldinger and F.~Halter-Koch, \emph{Non-{U}nique {F}actorizations.
  {A}lgebraic, {C}ombinatorial and {A}nalytic {T}heory}, Pure and Applied
  Mathematics, vol. 278, Chapman \& Hall/CRC, 2006.

\bibitem{Ge-Yu12b}
A.~Geroldinger and P.~Yuan, \emph{The set of distances in {K}rull monoids},
  Bull. Lond. Math. Soc. \textbf{44} (2012), 1203 �-- 1208.

\bibitem{HK98}
F.~Halter-Koch, \emph{Ideal {S}ystems. {A}n {I}ntroduction to {M}ultiplicative
  {I}deal {T}heory}, Marcel Dekker, 1998.

\bibitem{Ka16b}
F.~Kainrath, \emph{Arithmetic of {M}ori domains and monoids{\rm \,:} the
  {G}lobal {C}ase}, in \emph{Multiplicative {I}deal {T}heory and {F}actorization
  {T}heory}, Springer, 2016, pp.~183 -- 218.

\bibitem{N-P-T-W16a}
C.~O'Neill, V.~Ponomarenko, R.~Tate, and G.~Webb, \emph{On the set of catenary
  degrees of finitely generated cancellative commutative monoids}, Int. J.
  Algebra Comput. \textbf{26} (2016), 565 -- 576.

\bibitem{Sc09a}
W.A. Schmid, \emph{A realization theorem for sets of lengths}, J. Number Theory
  \textbf{129} (2009), 990 -- 999.

\end{thebibliography}
\end{document}